\documentclass[12pt,a4paper,fleqn]{article}
\usepackage{latexsym}
\usepackage[latin1]{inputenc}
\usepackage{amsmath}
\usepackage{amsfonts}
\usepackage{oldstyle}
\usepackage{stmaryrd}
\usepackage{color}
\usepackage{delarray}
\usepackage{eufrak}
\usepackage{amsthm}
\usepackage{amssymb}
\usepackage{euscript}
\newtheorem{prop}{Proposition}[section]
\newtheorem{cor}[prop]{Corollary}
\newtheorem{lemma}[prop]{Lemma}

\newtheorem{theorem}[prop]{Theorem}
\newtheorem{defi}[prop]{Definition}

\renewcommand{\geq}{\geqslant}
\def\leq{\leqslant}

\newcommand{\R}{\mathbb{R}}

\newcommand{\iot}{\int_{0}^{t}}

\newcommand{\cac}{\mathcal C}

\newcommand{\De}{\Delta}

\newcommand{\lp}{\left(}
\newcommand{\rp}{\right)}
\newcommand{\lc}{\left[}
\newcommand{\rc}{\right]}

\newcommand{\lln}{\left|}
\newcommand{\rrn}{\right|}

\def\cal{\mathcal}
\def\1{{\mathbf{1}}}

\def\1{{\mathbf{1}}}
\def\0.5{{\frac{1}{2}}}

\begin{document}

\title{Cross-variation of Young integral with respect to long-memory fractional Brownian motions}
\author{Ivan Nourdin\footnote{Facult\'e des Sciences, de la Technologie et de la Communication; UR en Math\'ematiques; Luxembourg University, 6 rue Richard Coudenhove-Kalergi, L-1359 Luxembourg; {\tt ivan.nourdin@uni.lu}} \ and Rola Zintout\footnote{Institut Elie Cartan, UMR 7502, Nancy Universit\'e - CNRS - INRIA; {\tt rola.zintout@univ-lorraine.fr}}}
\date{Universit\'e du Luxembourg and Universit\'e de Lorraine}

\maketitle

\noindent
{\bf Abstract}. We study the asymptotic behaviour of the cross-variation of two-dimensional processes having the form of a Young integral with respect to a fractional Brownian motion of index $H>\frac12$. When $H$ is smaller than or equal to $\frac34$, we show asymptotic mixed normality. When $H$ is stricly bigger than $\frac34$, we obtain a limit that is expressed in terms of the difference of two independent Rosenblatt processes.

\section{Introduction}

\subsection{Foreword and main results}

In the near past, there have been many applications of stochastic differential equations (SDE) driven by fractional Brownian motion  in different areas of mathematical modelling. To name but a few, we mention the use of such equations as a model for meteorological phenomena \cite{6,26}, protein dynamics \cite{52,52bis}, or noise in electrical networks \cite{41}.

Here, we consider more generally a two-dimensional stochastic process
$
\{X_t\}_{t\in[0,T]}=\{(X^{(1)}_t,X^{(2)}_t)\}_{t\in [0,T]}
$
of the form
\begin{equation}\label{X}
X_t^{(i)}=x_i+\iot \sigma^{i,1}_sdB^{(1)}_s+\iot \sigma^{i,2}_sdB^{(2)}_s,\quad t\in[0,T],\,i=1,2.
\end{equation}
In (\ref{X}), $B=(B^{(1)},B^{(2)})$ is a two-dimensional fractional Brownian motion of Hurst index $H>\frac12$ defined on a complete probability space $(\Omega,\mathcal{F},P)$, whereas $x=(x_1,x_2)\in\R^2$ and $\sigma$
is a $2\times 2$ matrix-valued process.
The case where $X$ solves a fractional SDE corresponds to $\sigma_t=\sigma(X_t)$, with $\sigma:\R^2\to\mathcal{M}_2(\R)$ deterministic.
Since
we are assuming that $H>\frac12$, by imposing appropriate conditions on $\sigma$ (see Section \ref{framework} for the details) we may and will assume throughout the text that
$\iot \sigma^{i,j}_sdB^{(j)}_s$ is understood
in the Young \cite{young} sense (see again Section \ref{framework} for the details).

In this paper, we are concerned with the asymptotic behaviour of the cross-variation associated to $X$ on $[0,T]$, which is
the sequence of stochastic processes defined as:
\begin{eqnarray}\label{JN}
J_n(t)=\sum_{k=1}^{\lfloor nt\rfloor}\De X_{k/n}^{(1)} \De X_{k/n}^{(2)}, \mbox{   }\mbox{   }n\geq 1,\,t\in [0,T].
\end{eqnarray}
Here, and the same anywhere else, we use the notation $\De X^{(i)}_{k/n}$ to indicate the increment $X^{(i)}_{k/n}-X^{(i)}_{(k-1)/n}$.
We shall show the following two theorems. They might be of interest for solving problems arising from statistics, as  for instance the problem of
testing the hypothesis $(H_0)$: ``$\sigma^{1,2}=\sigma^{2,1}=0$'' in (\ref{X}).
\begin{theorem}\label{maintheorem1}
For any $t\in[0,T]$,
\begin{equation}
n^{2H-1}J_n(t) \mbox{  } \overset{\rm prob}{\to} \int_0^t  (\sigma^{1,1}_s\sigma^{1,2}_s+\sigma^{2,1}_s\sigma^{2,2}_s)ds\quad
\mbox{as $n\rightarrow\infty$}.
\end{equation}
\end{theorem}
\begin{theorem}\label{maintheorem2}
Assume  $\sigma^{1,2}=\sigma^{2,1}=0$ and
let
\begin{equation}\label{an2}
a_n:=\left\{
\begin{array}{cll}
n^{2H-\frac12}&\quad&\mbox{if $\frac12<H<\frac34$}\\
\frac{n}{\sqrt{\log n}}&\quad&\mbox{if $H=\frac34$}\\
n&\quad&\mbox{if $\frac34<H<1$}
\end{array}
\right..
\end{equation}
Then, as $n\to\infty$,
\begin{equation}\label{convergence}
a_n
\, J_n \overset{\cal L}{\to} \int_0^\cdot \sigma ^{1,1}_s\sigma ^{2,2}_s dZ_s
\quad\mbox{in the Skorohod space $D[0,T]$}.
\end{equation}
In (\ref{convergence}), the definition of $Z$ is according to the value of $H$. More precisely,
$Z$ equals $\frac{C_H}2$ times $W$ when $H\in(\frac12, \frac34]$, with
$C_H$ given by (\ref{CH})-(\ref{CH2}) and $W$ a Brownian motion independent of $\mathcal{F}$;
and $Z=\frac12\big(R^{(1)}-R^{(2)}\big)$ when $H\in(\frac34,1)$, with $R^{(k)}$ the Rosenblatt process constructed from
the fractional Brownian motion
\[
\beta^{(k)}=\frac{1}{\sqrt{2}}(B^{(1)}+(-1)^{k+1}B^{(2)}), \quad k=1,2,
\]
see Definition \ref{rosenblatt} for the details.
\end{theorem}

\subsection{Link to the existing literature}

Our results are close in spirit to those contained in \cite{CNW} (which has been a strong source of inspiration to us), where central limit theorems for power variations of integral fractional processes are investigated.

As we will see our analysis of $J_n$, that requires similar but different efforts compared to \cite{CNW} (as we are here dealing with a  {\it two-dimensional} fractional Brownian motion on one hand and we also consider\footnote{The authors of \cite{CNW} did not consider the case where $H>\frac34$ since, quoting them, ``the problem is more involved because non-central limit theorems are required''.} the case where $H>\frac34$ on the other hand), is actually greatly simplified
by the use of a recent, nice result obtained in \cite{CNP} about the asymptotic behaviour of weighted random sums.

\subsection{Plan of the paper}

The rest of the paper is as follows. Section \ref{framework}
contains a thorough description of the framework in which our study takes place (in particular, we recall the definition of the Young integral and we provide its main properties).  Section \ref{preliminaries} gathers several preliminary results that will be essential for proving our main results. Finally, proofs of Theorems \ref{maintheorem1} and \ref{maintheorem2} are given in Section \ref{proof}.

\section{Our framework}\label{framework}

In this section, we describe the framework used throughout the paper and we fix a parameter $\alpha\in (0,1)$.

We let $C^\alpha$ denote the set of H\"older continuous functions of index $\alpha \in (0,1)$, that is, the set of those functions
$f: [0,T]\rightarrow \R$ satisfying
\begin{equation}\label{normf}
\lln f\rrn_{\alpha }:= \sup_{0\leq s<t\leq T}\frac{\lln f(t)-f(s)\rrn}{(t-s)^{\alpha}}<\infty.
\end{equation}
Also, we set
$\| f\|_{\alpha }:=\lln f\rrn_{\alpha }+ \lln f\rrn_{\infty }$,
with $|f|_\infty=\sup_{0\leq t\leq T}|f(t)|$.

For a fixed $f\in C^\alpha$, we consider the operator $T_f: \cac^1\rightarrow  \cac^1$ defined as
\[
T_f(g)(t)=\iot f(u) g'(u)du, \quad t\in [0,T].
\]
Let $\gamma\in(0,1)$ be such that $\alpha +\gamma>1$.
Then $T_f$ extends, in a unique way, to an operator
$T_f:C^\gamma\to C^\gamma$, which further satisfies
\[
\| T_f(g)\|_\gamma \leq \left(1+C_{\alpha,\gamma}\right) (1+T^\gamma)\|f\|_\alpha \|g\|_\gamma,
\]
with $C_{\alpha,\gamma}=\frac12 \sum_{n=1}^\infty 2^{-n(\alpha +\gamma-1)}<\infty$.
See, e.g., \cite[Theorem 3.1]{IN} for a proof.

\begin{defi}
Let $\alpha,\gamma\in(0,1)$ be such that $\alpha+\gamma>1$.
Let $f\in C^\alpha$ and $g\in C^\gamma$. The {\it Young integral} $\int_0^.f(u) dg(u)$ is then defined as being $T_f(g)$.
\end{defi}
The Young integral satisfies (see, e.g., \cite[inequality (3.3)]{IN}) that, for any $a,b\in[0,T]$ with $a<b$,
\begin{eqnarray}\label{young1}
\lln \int_a^b (f(u) -f(a))dg(u) \rrn\leq C_{\alpha,\gamma} | f |_{\alpha} | g |_{\gamma}\lp b-a \rp^{\alpha+\gamma}.
\end{eqnarray}

As we said in the Introduction, we let $B=(B^{(1)},B^{(2)})$ be a 2-dimensional fractional Brownian motion defined on a probability space $(\Omega,\mathcal{F},P)$. We assume further that $\mathcal{F}$ is the $\sigma$-field generated by $B$.
We also suppose that the Hurst parameter $H$ of $B$ is the same for the two components and that it
is strictly bigger than $\frac{1}{2}$.

Let $\alpha\in(0,1)$ and let $\sigma ^{i,j}:\Omega\times [0,T]\rightarrow \R$, $i,j=1,2$, be four given stochastic processes that are measurable with respect to $\mathcal{F}$.
We will assume throughout the text that the following two additional assumptions on $\alpha$ and $\sigma^{i,j}$ take place:\\

{\bf (A)} $\alpha\in\big(\frac14+\frac H2,H\big)$,\\

{\bf (B)} For each pair $(i,j)\in\{1,2\}^2$, the random variable $\| \sigma^{i,j}\|_{\alpha}$ has moments of all orders.\\

Observe that $\alpha+H>1$ due to both ({\bf A}) and $H>\frac12$, so that the integrals in (\ref{X}) are well-defined in the Young sense.
Also, recall the following variant of the Garcia-Rodemich-Rumsey Lemma \cite{GRR}: for any $q>1$, there exists a constant $c_{\alpha,q}>0$ (depending only on $\alpha$ and $q$) such that
\begin{equation}\label{grr}
|B^{(i)}|_\alpha^{q}\leq c_{\alpha,q}\iint_{[0,T]^2}\frac{|B^{(i)}_u-B^{(i)}_v|^q}{|u-v|^{2+q\alpha}}dudv.
\end{equation}
Using (\ref{grr}), one deduces that $|B^{(i)}|_\alpha$ has
moments of all orders.

\section{Preliminaries}\label{preliminaries}

\subsection{Breuer-Major theorem}
The next statement is a direct consequence of the celebrated Breuer-Major \cite{breuermajor} theorem (see  \cite[Section 7.2]{IN} for a modern proof). We write `fdd' to indicate the convergence of all the finite-dimensional distributions.

\begin{theorem}[Breuer-Major]\label{BM-thm} Let $\beta$ be a (one-dimensional) fractional Brownian motion of index $H\in(0,\frac34]$. Then, as $n\to\infty$ and with $W$ a standard Brownian motion,
\begin{enumerate}
\item[(i)] if $H<\frac34$ then
\begin{eqnarray*}
&&\left\{\frac{1}{\sqrt{n}}\sum_{k=1}^{\lfloor n t\rfloor} \big[(\beta_{k}-\beta_{k-1})^2-1\big]\right\}_{t\in [0,T]}\\
&\overset{\rm fdd}{\longrightarrow}&\frac{1}{2}\sum_{k\in\mathbb{Z}} \big(
|k+1|^{2H}+|k-1|^{2H}-2|k|^{2H}
\big)^2\,
\{W_t\}_{t\in [0,T]};
\end{eqnarray*}
\item[(ii)] if $H=\frac34$ then
\begin{eqnarray*}
\left\{\frac{1}{\sqrt{n\log n}}\sum_{k=1}^{\lfloor n t\rfloor} \big[(\beta_{k}-\beta_{k-1})^2-1\big]\right\}_{t\in [0,T]}
\overset{\rm fdd}{\longrightarrow}\frac{3}{4}\log 2\,
\{W_t\}_{t\in [0,T]}.
\end{eqnarray*}
\end{enumerate}
\end{theorem}

By a scaling argument (to pass from $k$ to $k/n$) and by using the seminal result of Peccati and Tudor \cite{PT} (to allow an extra $F$),
one immediately deduces from Theorem \ref{BM-thm} the following corollary.

\begin{cor}\label{BM-cor}
Let $\beta=(\beta^{(1)},\beta^{(2)})$ be a two-dimensional fractional Brownian motion of index $H\in(0,\frac34]$. Then, as $n\to\infty$ and with $W$ a (one-dimensional) standard Brownian motion independent of $\beta$, we have, for any random vector $F=(F_1,\ldots,F_d)$ measurable
with respect to $\beta$,
\begin{enumerate}
\item[(i)] if $H<\frac34$ then
\begin{eqnarray*}
&&\left\{F, n^{2H-\frac12}
\sum_{k=1}^{\lfloor n t\rfloor} \big[
(\beta^{(1)}_{k/n}-\beta^{(1)}_{(k-1)/n})^2-
(\beta^{(2)}_{k/n}-\beta^{(2)}_{(k-1)/n})^2\big]
\right\}_{t\in [0,T]}\\
&&\overset{\rm fdd}{\longrightarrow}
\left\{F,C_H\,
W_t\right\}_{t\in [0,T]},
\end{eqnarray*}
where
\begin{equation}\label{CH}
C_H=\frac{1}{\sqrt{2}}\sum_{k\in\mathbb{Z}} \big(
|k+1|^{2H}+|k-1|^{2H}-2|k|^{2H}
\big)^2
\end{equation}
\item[(ii)] if $H=\frac34$ then
\begin{eqnarray*}
&&\left\{F, \frac{n}{\sqrt{\log n}}
\sum_{k=1}^{\lfloor n t\rfloor} \big[
(\beta^{(1)}_{k/n}-\beta^{(1)}_{(k-1)/n})^2-
(\beta^{(2)}_{k/n}-\beta^{(2)}_{(k-1)/n})^2\big]
\right\}_{t\in [0,T]}\\
&&\overset{\rm fdd}{\longrightarrow}
\left\{F,C_{3/4}\,W_t\right\}_{t\in [0,T]},
\end{eqnarray*}
where
\begin{equation}\label{CH2}
C_{3/4}=\frac{3\sqrt{2}}{4}\log 2.
\end{equation}
\end{enumerate}
\end{cor}

\subsection{Taqqu's theorem and the Rosenblatt process}

Taqqu's theorem \cite{taqqu} describes the fluctuations
of the quadratic variation of the fractional Brownian motion
when the Hurst index $H$ is strictly bigger than $\frac34$, that is, for the range of values which are not covered by the Breuer-Major Theorem \ref{BM-thm}.
We state here a version that fits into our framework.
With respect to the original statement, it is worthwhile noting
that, in Theorem \ref{taqqu} (whose proof may be found in \cite{NNT}),  the convergence is in $L^2(\Omega)$ (and not only in law).
This latter fact will reveal to be crucial in our proof of Theorem \ref{maintheorem2}, as it will allow us to apply the main result of \cite{CNP} recalled in Section \ref{sec-CNP}.

\begin{theorem}[Taqqu]\label{taqqu} Let $\beta$ be a (one-dimensional) fractional Brownian motion of index $H\in(\frac34,1)$. Then, for any $t\in[0,T]$, the sequence
\begin{eqnarray}\label{ds}
n^{1-2H}\sum_{k=1}^{\lfloor n t\rfloor} \big[n^{2H}(\beta_{k/n}-\beta_{(k-1)/n})^2-1\big]
\end{eqnarray}
converges in $L^2(\Omega)$  as $n\to\infty$.
\end{theorem}

\begin{defi}\label{rosenblatt}
Let the assumption of Theorem \ref{taqqu} prevail
and denote by $R_t$ the limit of (\ref{ds}).
The process $R=\{R_t\}_{t\in [0,T]}$ is called the Rosenblatt process constructed from $\beta$.
\end{defi}

For the main properties of the Rosenblatt process $R$, we refer the reader to Taqqu \cite{rosenblatt} or Tudor \cite{tudor}. See also \cite[Section 7.3]{IN}.
An immediate corollary of Theorem \ref{taqqu} is as follows.
\begin{cor}\label{taqqu-cor}
Let $\beta=(\beta^{(1)},\beta^{(2)})$ be a two-dimensional fractional Brownian motion of index $H\in(\frac34,1)$. Then, for any $t\in [0,T]$,
\begin{eqnarray*}
n
\sum_{k=1}^{\lfloor n t\rfloor} \big[
(\beta^{(1)}_{k/n}-\beta^{(1)}_{(k-1)/n})^2-
(\beta^{(2)}_{k/n}-\beta^{(2)}_{(k-1)/n})^2\big]
\overset{L^2(\Omega)}{\longrightarrow}
R^{(1)}_t-R^{(2)}_t
\end{eqnarray*}
as $n\to\infty$, where $R^{(i)}$ is the Rosenblatt process constructed from
the fractional Brownian motion $\beta^{(i)}$, $i=1,2$.
\end{cor}

\subsection{Two simple auxiliary lemmas}
To complete the proofs of Theorems \ref{maintheorem1} and \ref{maintheorem2} we will, among other things, need the following two simple lemmas.
\begin{lemma}\label{sigma1}
Let $B$ and $\sigma$ be as in Section \ref{framework}.
Then there exists a constant $C=C(\alpha,H,T,\sigma)>0$ such that, for any $i,j=1,2$, any  $n\geq 1$ and any $k\in\{1,...,\lfloor nT\rfloor\}$,
\begin{eqnarray}\label{claim1}
\left\|\int_{(k-1)/n}^{k/n}(\sigma ^{i,j}_s-\sigma ^{i,j}_{k/n})dB^j_s\right\|_{L^2(\Omega)}&\leq& Cn^{-2\alpha },\\
\label{claim2}
 \left\|\int_{(k-1)/n}^{k/n}\sigma ^{i,j}_sdB^j_s\right\|_{L^2(\Omega)}&\leq& C n^{-H}.
\end{eqnarray}
\end{lemma}

\begin{proof}
Without loss of generality, we may and will assume that $i=j=1$. Using (\ref{young1}) with $\beta=\alpha$, we have, almost surely,
\begin{eqnarray*}
\lln\int_{(k-1)/n}^{k/n}\lp \sigma ^{1,1}_s - \sigma ^{1,1}_{k/n}\rp dB^1_s\rrn&\leq& C_{\alpha,\alpha}|\sigma ^{1,1}|_{\alpha}
| B^1|_{\alpha} n^{-2\alpha}.
\end{eqnarray*}
Using Cauchy-Schwarz inequality, one deduces
\begin{eqnarray*}
&& E\lc\lp\int_{(k-1)/n}^{k/n}\lp \sigma ^{1,1}_s - \sigma ^{1,1}_{k/n}\rp dB^1_s\rp^2\rc\\
&\leq & C^2_{\alpha,\alpha}\sqrt{E\lc\|\sigma ^{1,1}\|_{\alpha}^4\rc}\sqrt{\lc E| B^1|_{\alpha}^4 \rc}\,n^{-4\alpha} =Cn^{-4\alpha},
\end{eqnarray*}
thus yielding (\ref{claim1}).
On the other hand, one has
\begin{eqnarray*}
&& \left\| \int_{(k-1)/n}^{k/n}\sigma ^{i,j}_sdB^j_s\right\|_{L^2(\Omega)}\\
 &\leq&  \left\|\int_{(k-1)/n}^{k/n}\lp \sigma ^{i,j}_s -
\sigma ^{i,j}_{k/n}\rp dB^j_s\right\|_{L^2(\Omega)}
\!\!\!\!+\bigg\|\sigma ^{i,j}_{k/n} \De B^j_{k/n}\bigg\|_{L^2(\Omega)}\\
&\leq& Cn^{-2\alpha}+Cn^{-2H},\quad \mbox{by (\ref{claim1}) and because of ({\bf B})}\\
&\leq&Cn^{-H},\quad\mbox{using ({\bf A})},
\end{eqnarray*}
which is the desired claim (\ref{claim2}).
\end{proof}

\begin{lemma}\label{neuenkirch}
Let $g,h:[0,T]\to\R$ be two continuous functions, let $\gamma\in\R$, and let us write $\De h_{k/n}$ to denote the increment $h(k/n)-h((k-1)/n)$. If
\begin{equation}\label{hyp-lm}
\forall t\in[0,T]\cap \mathbb{Q}:\quad \lim_{n\to\infty}n^{\gamma}\sum_{k=1}^{\lfloor nT\rfloor}{\bf{1}}_{[0,t]}(k/n) \lp\De h_{k/n}\rp^2=t,
\end{equation}
then, for all $t\in [0,T]$,
\begin{eqnarray*}
\lim_{n\to\infty} n^{\gamma}\sum_{k=1}^{\lfloor nT\rfloor}g(k/n) {\bf{1}}_{[0,t]}\lp k/n\rp \lp\De h_{k/n}\rp^2=\iot g(s)ds.
\end{eqnarray*}
\end{lemma}
\begin{proof}
Since $t\mapsto n^{\gamma}\sum_{k=1}^{n}{\bf{1}}_{[0,t]}(k/n) \lp\De h_{k/n}\rp^2$
is non-decreasing, it is straightforward to deduce from (\ref{hyp-lm}) that, for {\it all} $t\in[0,T]$,
\[
\lim_{n\to\infty}n^{\gamma}\sum_{k=1}^{\lfloor nT\rfloor}{\bf{1}}_{[0,t]}(k/n) \lp\De h_{k/n}\rp^2=t.
\]
Otherwise stated, the cumulative distribution function (cdf) of the compactly supported
measure
\[
\nu_n(dx)=n^{\gamma}\sum_{k=1}^{\lfloor nT\rfloor} \lp\De h_{k/n}\rp^2\delta_{k/n}(dx),
\]
where $\delta_a$ stands for the Dirac mass at $a$,
converges pointwise to the cdf of the Lebesgue measure on $[0,T]$. Since $g$ is continuous, it is then a routine exercise to deduce that our desired claim holds true.
\end{proof}
\subsection{Asymptotic behaviour of weighted random sums, following Corcuera, Nualart and Podolskij \cite{CNP}}\label{sec-CNP}

The following result represents a central ingredient in the proof of both Theorems \ref{maintheorem1} and \ref{maintheorem2}.

\begin{prop}\label{auxiliary}
Let $u=\{u_t\}_{t\in [0,T]}$ be a H\"older continuous process with index $\alpha>\frac12$, set
\[
K_n(t)=\sum_{k=1}^{\lfloor nt\rfloor} u_{k/n}\De B^{(1)}_{k/n}
\De B^{(2)}_{k/n},\quad t\in [0,T],
\]
and let $a_n$ be given by (\ref{an2}).
Then, as $n\to\infty$,
\begin{equation}\label{vf}
a_n\, K_n \overset{\cal L}{\to}\mbox{  } \int_0^\cdot u_s dZ_s\quad\mbox{in the Skorohod space $D[0,T]$}.
\end{equation}
Here, $Z$ is as in the statement of Theorem \ref{maintheorem2}.
\end{prop}

The proof of our Proposition \ref{auxiliary} heavily relies on
a nice result taken from Corcuera, Nualart and Podolskij \cite{CNP}. Actually, we will need a slight extension of the result of \cite{CNP}, that we state here for convenience (and also because we do not share the same notation). The only difference between Theorem \ref{thm-CNP} as stated below and its original version appearing in \cite{CNP} is that $Z$ need not be a Brownian motion. A careful inspection of the proof given in \cite{CNP} indeed reveals that the Brownian feature of $Z$ plays actually no role; the only property of $Z$ which is used is that the sum of its H\"older exponent and that of $u$ is strictly bigger than 1, see ${\bf (H1)}$.

\begin{theorem}[Corcuera, Nualart, Podolskij]\label{thm-CNP}
The underlying probability space is $(\Omega,\mathcal{F},P)$. Let $u=\{u_t\}_{t\in [0,T]}$ be a H\"older continuous process with index $\alpha\in(0,1)$, and let $\xi=\{\xi_{k,n}\}_{n\in\mathbb{N},\,1\leq k\leq \lfloor nT\rfloor}$ be a family of random variables. Set \[
g_n(t)=\sum_{k=1}^{\lfloor nt\rfloor}\xi_{k,n},\quad t\in [0,T].
\]
Assume the following two hypotheses on the double sequence $\xi$:
\begin{itemize}
\item[{\bf (H1)}] $\{g_n(t)\}_{t\in [0,T]}\overset{\rm f.d.d.}{\to}\{Z(t)\}_{t\in [0,T]}$ $\mathcal{F}$-stably, where $Z$ is
H\"older continuous with index $\beta$ such that $\alpha+\beta>1$.
\item[{\bf (H2)}] There is a constant $C>0$ such that, for any $1\leq i<j\leq [nT]$,
\[
E\left[\left(\sum_{k=i+1}^j\xi_{k,n}\right)^4\right]\leq C\left(\frac{j-i}{n}\right)^2.
\]
\end{itemize}
Then
\[
\sum_{k=1}^{\lfloor n\cdot\rfloor} u_{\frac{k}n}\,\xi_{k,n} \overset{\cal L}{\to}\int_0^\cdot u_s dZ_s
\quad\mbox{in the Skorohod space $D[0,T]$},
\]
where $\int_0^\cdot u_s dZ_s$ is understood as a Young integral.
\end{theorem}

Armed with Theorem \ref{thm-CNP}, we are now ready to prove Proposition \ref{auxiliary}.

\begin{proof}[Proof of Proposition \ref{auxiliary}]
Set $\xi_{k,n}=a_n\De B^{(1)}_{k/n}
\De B^{(2)}_{k/n}$ and
$g_n(t)=\sum_{k=1}^{\lfloor nt\rfloor}\xi_{k,n}$, $t\in [0,T]$.
We shall check the two assumptions ${\bf (H1)}$ and
${\bf (H2)}$ of  Theorem \ref{thm-CNP}.\\

\underline{\it Step 1: Checking $(\bf{H1})$}. We make use of the rotation trick. More precisely, 
let
 $\beta^{(1)}=\frac{1}{\sqrt{2}}(B^{(1)}+B^{(2)})$ and $\beta^{(2)}=\frac{1}{\sqrt{2}}(B^{(1)}-B^{(2)})$,
 so that $\xi_{k,n}=\frac{a_n}2\left(\big(\Delta\beta^{(1)}_{k/n}\big)^2-\big(\Delta\beta^{(2)}_{k/n}\big)^2\right)$.
 It is easy to check that $\beta^{(1)}$ and $\beta^{(2)}$ are
 two independent fractional Brownian motions of index $H$.
 As a result, assumption $(\bf{H1})$ is satisfied thanks to 
Corollary \ref{BM-cor} (resp. Corollary \ref{taqqu-cor}) when $H\leq\frac34$ (resp.  $H>\frac34$).\\

\underline{\it Step 2: Checking $(\bf{H2})$}. Since all the $L^p(\Omega)$-norms are equivalent inside a given Wiener chaos (here: the second Wiener chaos), it suffices to check the existence of a constant $C>0$ such that, for any $1\leq i<j\leq [nT]$,
\begin{equation}\label{H2-hyp}
E\left[\left(\sum_{k=i+1}^j\xi_{k,n}\right)^2\right]\leq C\,\frac{j-i}{n}.
\end{equation}
Using the independence of $B^{(1)}$ and $B^{(2)}$, one computes that
\[
E\left[\left(\sum_{k=i+1}^j\xi_{k,n}\right)^2\right]=a_n^2\,n^{-4H}\sum_{k,k'=i+1}^j \rho(k-k')^2,
\]
with $\rho(r)=\frac12\big(|r+1|^{2H}+|r-1|^{2H}-2|r|^{2H}\big)$.
As a result, for any $1\leq i<j\leq [nT]$,
\[
E\left[\left(\sum_{k=i+1}^j\xi_{k,n}\right)^2\right]\leq a_n^2\,n^{-4H}(j-i)\sum_{r=-[nT]}^{[nT]} \rho(r)^2.
\]
It is straightforward to show that $a_n^2\,n^{1-4H}\sum_{r=-[nT]}^{[nT]} \rho(r)^2=O(1)$ as $n\to\infty$.
Thus, (\ref{H2-hyp}) is satisfied, and so is $(\bf{H2})$.\\

To conclude the proof of Proposition \ref{auxiliary}, it remains to apply Theorem  \ref{thm-CNP} with $\xi_{k,n}=a_n\De B^{(1)}_{k/n}
\De B^{(2)}_{k/n}$.
\end{proof}
\section{Proof of our main results}\label{proof}

\subsection{Proof of Theorem \ref{maintheorem1}}
We divide it into several steps.

\bigskip
\underline{\it Step 1}. Recall $J_n$ from (\ref{JN}). One can write

\begin{eqnarray*}
J_n(t)&=&\sum_{k=1}^{\lfloor n t\rfloor} \lp\int_{(k-1)/n}^{k/n}\sigma ^{1,1}_s dB^1_s + \int_{(k-1)/n}^{k/n}\sigma ^{1,2}_s dB^2_s\rp\\
& &\hskip2cm\times \lp\int_{(k-1)/n}^{k/n}\sigma ^{2,1}_s dB^1_s + \int_{(k-1)/n}^{k/n}\sigma ^{2,2}_s dB^2_s\rp\\
&=:& A_{n}(t) + R_{1,n}(t) + R_{2,n}(t),
\end{eqnarray*}
with
\begin{eqnarray}
A_{n}(t)&=&\sum_{k=1}^{\lfloor n t\rfloor} \lp\sigma ^{1,1}_{k/n} \De B^1_{k/n} +\sigma ^{1,2}_{k/n} \De B^2_{k/n}\rp
\lp\sigma ^{2,1}_{k/n} \De B^1_{k/n} +\sigma ^{2,2}_{k/n} \De B^2_{k/n}\rp,\notag\\
\label{an}\\
R_{1,n}(t)&=&\sum_{k=1}^{\lfloor n t\rfloor} \lp \int_{(k-1)/n}^{k/n} \sigma ^{1,1}_s dB^1_s +\int_{(k-1)/n}^{k/n} \sigma ^{1,2}_s  dB^2_s  \rp\label{r1n}\\
&&\hskip1cm\times
 \lp \int_{(k-1)/n}^{k/n}\lp \sigma ^{2,1}_s - \sigma ^{2,1}_{k/n}\rp dB^1_s +
\int_{(k-1)/n}^{k/n}\lp \sigma ^{2,2}_s - \sigma ^{2,2}_{k/n}\rp dB^2_s  \rp,\notag\\
R_{2,n}(t)&=&\sum_{k=1}^{\lfloor n t\rfloor} \lp\sigma ^{2,1}_{k/n} \De B^1_{k/n} +\sigma ^{2,2}_{k/n} \De B^2_{k/n}\rp \label{r2n}\\
& &\hskip1cm\times\lp \int_{(k-1)/n}^{k/n}\lp \sigma ^{1,1}_s - \sigma ^{1,1}_{k/n}\rp dB^1_s +
\int_{(k-1)/n}^{k/n}\lp \sigma ^{1,2}_s - \sigma ^{1,2}_{k/n}\rp dB^2_s  \rp.\notag
\end{eqnarray}

\underline{\it Step 2}.
Let us prove the convergence of $n^{2H-1}R_{i,n}(t)$, $i=1,2$, $t\in [0,T]$, in $L^1(\Omega)$ towards zero.
Using Cauchy-Schwarz and Lemma \ref{sigma1}, we see that
\begin{eqnarray*}
\|R_{1,n}(t)\|_{L^1(\Omega)}&\leq&\sum_{k=1}^{\lfloor n t\rfloor} \left\|  \int_{(k-1)/n}^{k/n}\sigma ^{1,1}_s dB^1_s +
\int_{(k-1)/n}^{k/n}\sigma ^{1,2}_s dB^2_s\right\|_{L^2(\Omega)}\\
&\times&\left\| \int_{(k-1)/n}^{k/n}\lp \sigma ^{2,1}_s - \sigma ^{2,1}_{k/n}\rp dB^1_s
+\int_{(k-1)/n}^{k/n}\lp \sigma ^{2,2}_s - \sigma ^{2,2}_{k/n}\rp dB^2_s  \right\|_{L^2(\Omega)}\\
&\leq&C n^{-(H+2\alpha-1)}.
\end{eqnarray*}
Thanks to our assumption ({\bf A}), one deduces that $n^{2H-1}\|R_{1,n}(t)\|_{L^1(\Omega)}\rightarrow0$ as $n\rightarrow\infty$.
Similarly, one proves that  $n^{2H-1}\|R_{2,n}(t)\|_{L^1(\Omega)}\rightarrow0$.\\

\underline{\it Step 3}. Let us now consider $A_{n}$. One has
\begin{eqnarray*}
A_{n}(t)&=&\sum_{k=1}^{\lfloor n t\rfloor} \lp\sigma ^{1,1}_{k/n} \De B^1_{k/n} +\sigma ^{1,2}_{k/n} \De B^2_{k/n}\rp
\lp\sigma ^{2,1}_{k/n} \De B^1_{k/n} +\sigma ^{2,2}_{k/n} \De B^2_{k/n}\rp\\
&=:& A_{1,n}(t)+A_{2,n}(t)+S_{n}(t),
\end{eqnarray*}
with
\begin{eqnarray}
A_{i,n}(t)&=&\sum_{k=1}^{\lfloor n t\rfloor} \sigma ^{1,i}_{k/n} \sigma ^{2,i}_{k/n}\lp\De B^i_{k/n}\rp^2,\quad i=1,2,\label{ain}\\
S_{n}(t)&=&\sum_{k=1}^{\lfloor n t\rfloor} \lp\sigma ^{1,1}_{k/n} \sigma ^{2,2}_{k/n}+\sigma ^{1,2}_{k/n} \sigma ^{2,1}_{k/n}\rp\De B^1_{k/n}
\De B^2_{k/n}.\label{r3n}
\end{eqnarray}
Using Proposition \ref{auxiliary} and whatever the value of $H$ compared to $\frac{3}{4}$, one immediately checks that
$n^{2H-1}S_{n}(t)$ converges in law to zero, thus in probability.
On the other hand, fix $i\in\{1,2\}$ and recall the well-known fact that, for any $t\in[0,T]$,
\[
\lim_{n\to\infty}n^{2H-1}\sum_{k=1}^{\lfloor n T\rfloor}{\bf{1}}_{[0,t]}(k/n) \lp\De B^i_{k/n}\rp^2=t\quad\mbox{almost surely}.
\]
We then deduce that,  with probability 1, assumption (\ref{hyp-lm}) holds true with $h=B^i$ and $\gamma=2H-1$.
Lemma \ref{neuenkirch} applies and yields that
\begin{eqnarray*}
n^{2H-1}A_{i,n}(t)\rightarrow
\int_0^t\sigma ^{1,i}_s\sigma ^{2,i}_s ds\quad\mbox{almost surely}.
\end{eqnarray*}

\underline{\it Step 4}. Plugging together the conclusions of Steps 1 to 3 completes the proof of Theorem \ref{maintheorem1}.\qed

\subsection{Proof of Theorem \ref{maintheorem2}}
Recall from the previous section that
$J_n=A_{1,n}+A_{2,n}+S_n+R_{1,n}+R_{2,n}$, with
$A_{i,n}$, $S_n$, $R_{1,n}$ and $R_{2,n}$ given by (\ref{ain}), (\ref{r3n}), (\ref{r1n}) and (\ref{r2n}) respectively.
Using the estimates of Step 2 in the previous section, we easily obtain that, under ({\bf A}),
$a_n\, R_{i,n}(t)$
tends to zero in $L^1(\Omega)$ as $n\to\infty$, $i=1,2$, $t\in [0,T]$.
Moreover, the quantities $A_{1,n}$ and $A_{2,n}$
given by (\ref{ain})
equal zero when $\sigma^{1,2}=\sigma^{2,1}=0$. As a result, the asymptotic behavior of
$a_n\, J_{n}$
is the same as that of
$a_n\,S_{n}$,
and the desired conclusion follows directly from Proposition \ref{auxiliary}.
\qed

\bigskip

\bigskip
\noindent
{\bf Acknowledgments}. We thank David Nualart and Mark Podolskij for helpful discussions about reference \cite{CNP}.
We also thank an anonymous referee for his/her careful reading and valuable suggestions.

\end{document}